\documentclass{amsart}

\theoremstyle{plain}
\newtheorem{theorem}{Theorem}
\newtheorem{claim}{Claim}

\title{Toroidal Seifert fibered surgeries\\ on alternating knots}

\author[K.~Ichihara]{Kazuhiro Ichihara}
\address{Department of Mathematics, College of Humanities and Sciences, Nihon University,
3-25-40 Sakurajosui, Setagaya-ku, Tokyo 156-8550, Japan}
\email{ichihara@math.chs.nihon-u.ac.jp}

\author{In Dae Jong}
\address{Department of Mathematics, Kinki University, 3-4-1 Kowakae, Higashiosaka City, Osaka 577-0818, Japan} 
\email{jong@math.kindai.ac.jp}

\dedicatory{Dedicated to Professor Yoshihiko Marumoto for his 60th birthday}

\date{\today}



\begin{document}

\maketitle

\begin{abstract}
We give a complete classification of 
toroidal Seifert fibered surgeries on alternating knots. 
Precisely, we show that 
if an alternating knot $K$ admits a toroidal Seifert fibered surgery, 
then $K$ is either the trefoil knot and the surgery slope is zero, or 
the connected sum of a $(2,p)$-torus knot and a $(2,q)$-torus knot 
and the surgery slope is $2(p+q)$ with $|p|, |q| \ge 3$. 
\end{abstract}

\section{Introduction}

The hyperbolic Dehn surgery theorem, 
due to Thurston \cite[Theorem 5.8.2]{Thurston1978}, 
states that 
all but finitely many Dehn surgeries on a hyperbolic knot yield hyperbolic manifolds. 
Here a knot is called \textit{hyperbolic} 
if its complement admits a complete hyperbolic structure of finite volume. 
In view of this, 
a Dehn surgery on a hyperbolic knot yielding a non-hyperbolic manifold 
is called \textit{exceptional}. 
As a consequence of the Geometrization Conjecture, 
raised by Thurston~\cite[Section 6, question 1]{Thurston1982}, 
and established by celebrated Perelman's works~\cite{Perelman2002,Perelman2003,Perelman2003a}, 
exceptional surgeries are classified into \textit{Seifert fibered surgeries}, \textit{toroidal surgeries} or \textit{reducible surgeries}. 
We refer the reader to~\cite{Boyer2002} for a survey. 

Here we note that the classification is not exclusive, 
for there exist Seifert fibered 3-manifolds which are toroidal or reducible. 
However, a hyperbolic knot in the 3-sphere $S^3$ 
is conjectured to admit no reducible surgery. 
This is the Cabling Conjecture~\cite{Gonzalez-AcunaShort1986} which is well known but still open. 
Thus, we consider in this paper a Dehn surgery on a knot in $S^3$ 
yielding a $3$-manifold which is toroidal and Seifert fibered, 
called a \textit{toroidal Seifert fibered surgery}. 

It was shown that there exist infinitely many hyperbolic knots in $S^3$ 
each of which admits a toroidal Seifert fibered surgery 
by Eudave--Mu\~{n}oz~\cite[Proposition 4.5 (1) and (3)]{Eudave-Munoz2002}, 
and Gordon and Luecke~\cite{GordonLuecke1999a} independently. 
On the other hand, Motegi~\cite{Motegi2003} studied 
toroidal Seifert fibered surgeries on symmetric knots, and 
gave several restrictions on the existence of such surgeries. 
In particular, he showed that just the trefoil knot 
admits a toroidal Seifert fibered surgery 
among two-bridge knots~\cite[Corollary 1.6]{Motegi2003}.
Furthermore the authors showed that 
if a Montesinos knot admits a toroidal Seifert fibered surgery, 
then the knot is the trefoil knot and the surgery slope is zero~\cite{IchiharaJong2010}.

In this paper, we show the following. 

\begin{theorem}\label{thm:toralt}
If an alternating knot $K$ admits a toroidal Seifert fibered surgery, 
then $K$ is either the trefoil knot and the surgery slope is zero, or 
the connected sum of a $(2,p)$-torus knot and a $(2,q)$-torus knot 
and the surgery slope is $2(p+q)$. 
Here $p$ and $q$ are odd integers with $|p|, |q| \ge 3$. 
\end{theorem}

We note that Theorem~\ref{thm:toralt} for hyperbolic alternating knots also follows from 
a complete classification of exceptional surgeries on hyperbolic alternating knots 
recently achieved by the first author and Masai~\cite{IchiharaMasai}. 
While the classification is established by heavy computer-aided calculations, 
the proof given in this paper is quite simpler and direct.

\subsection*{Acknowledgments}
The first author is partially supported by
Grant-in-Aid for Young Scientists (B), No.~20740039,
Ministry of Education, Culture, Sports, Science and Technology, Japan.

\section{Proof}

We start with recalling definitions and basic facts. 

A knot in the $3$-sphere $S^3$ is called \textit{alternating} 
if it admits a diagram with alternately arranged 
over-crossings and under-crossings running along it. 
Menasco~\cite{Menasco1984} showed that 
an alternating knot is hyperbolic unless it is the connected sum of them or a $(2,p)$-torus knot. 

Let $K$ be a knot in $S^3$ and $E(K)$ the exterior of $K$. 
A {\it slope} on the boundary torus $\partial E(K)$ is an isotopy class of a non-trivial simple closed curve on $\partial E(K)$.  
For a slope $\gamma$ on $\partial E(K)$, 
we denote by $K(\gamma)$ the $3$-manifold obtained by Dehn surgery on $K$ along the slope $\gamma$, 
i.e., $K(\gamma)$ is obtained by gluing a solid torus $V$ 
to $E(K)$ so that a simple closed curve representing $\gamma$ bounds a meridian disk in $V$. 
We call such a slope $\gamma$ the {\it surgery slope}. 
It is well known that 
a slope on $\partial E(K)$ is parameterized by 
an element of $\mathbb{Q} \cup \{ 1/0 \}$ 
by using the standard meridian-longitude system for $K$. 
Thus, when a slope $\gamma$ corresponds to $r \in \mathbb{Q} \cup \{1/0\}$, 
we call the Dehn surgery along $\gamma$ the \textit{$r$-surgery} for brevity, and 
denote $K(\gamma)$ by $K(r)$. 
See~\cite{RolfsenBook} for basic references.

\begin{proof}[Proof of Theorem~\ref{thm:toralt}]

Now we start the proof of Theorem~\ref{thm:toralt} which will be achieved by the following two claims. 

\begin{claim}\label{clm:alt}
If a prime alternating knot $K$ admits a toroidal Seifert fibered surgery, 
then $K$ is the trefoil knot and the surgery slope is zero. 
\end{claim}
\begin{proof}
Let $K$ be a prime alternating knot such that $K(r)$ is a toroidal Seifert fibered $3$-manifold. 
Then $K$ is either a two-bridge knot or an alternating pretzel knot of length three, see~\cite[Lemma 3.1]{BoyerZhang1996}, \cite[p13]{Patton1995}. 

If $K$ is a two-bridge knot, 
then $K$ must be the trefoil knot and $r=0$~\cite[Corollary 1.6]{Motegi2003}. 

Assume that $K$ is an alternating pretzel knot of length three $P(a,b,c)$ which is not a two-bridge knot. 
Then $K$ is a small knot~\cite{Oertel1984}. 
Therefore $K$ must be fibered and $r=0$~\cite[Proposition 1]{IchiharaMotegiSong2008}. 
If the integers $a$, $b$, and $c$ are odd, then $P(a,b,c)$ is a genus one knot. 
This contradicts to the assumption that $K$ is not a two-bridge knot since the genus one fibered knots are just the trefoil knot and the figure-eight knot. 
If one of the integers $a$, $b$ and $c$ is even and the others are odd, 
then the surgery slope $r$ is a boundary slope 
of a non-orientable surface with crosscap number two~\cite{BoyerZhang1996}, \cite{Patton1995}. 
Without loss of generality, we may assume that $a$ is even and $b,c$ are odd. 
Then we have $r=2(b+c)$. 
Since $K$ is alternating, the sign of $b$ coincides with that of $c$. 
Therefore we have $r = 2(b+c) \ne 0$. 
This contradicts the condition $r=0$. 
\end{proof}

\begin{claim}\label{clm:alt-comp}
If a composite alternating knot $K$ admits a toroidal Seifert fibered surgery, 
then $K$ is the connected sum of a $(2,p)$-torus knot and a $(2,q)$-torus knot, 
and the surgery slope is $2(p+q)$. 
Here $p$ and $q$ are integers with $|p|, |q| \ge 3$. 
\end{claim}
\begin{proof}
According to the classification of non-simple Seifert fibered surgeries 
on non-hyperbolic knots~\cite[Theorem 1.2]{MiyazakiMotegi1997}, 
a composite alternating knot admitting a toroidal Seifert fibered surgery 
is just the connected sum of a $(2,p)$-torus knot and a $(2,q)$-torus knot, 
and the surgery slope is $2(p+q)$ with $|p|, |q| \ge 3$. 
\end{proof}

By Claims~\ref{clm:alt} and \ref{clm:alt-comp}, the proof of Theorem~\ref{thm:toralt} has completed. 
\end{proof}


\begin{thebibliography}{99}


\bibitem{Boyer2002}
S. Boyer, 
\textit{Dehn surgery on knots}, 
Handbook of Geometric topology, North-Holland, Amsterdam, 2002, ch. 4, 165--218. 

\bibitem{BoyerZhang1996}
S. Boyer and X. Zhang, 
\textit{Cyclic surgery and boundary slopes}, 
Geometric Topology, W. Kazez ed., AMS/IP Studies in Advanced Mathematics 2 (1996), 62--79. 

\bibitem{Eudave-Munoz2002}
M. Eudave-Mu\~{n}oz, 
\textit{On hyperbolic knots with Seifert fibered Dehn surgeries}, 
Topology Appl. \textbf{121} (2002), no.1--2, 119--141.

\bibitem{Gonzalez-AcunaShort1986}
F. Gonz\'alez-Acu\~{n}a and H. Short, 
\textit{Knot surgery and primeness}, 
Math. Proc. Cambridge Philos. Soc. \textbf{99} (1986), no.1, 89--102.

\bibitem{GordonLuecke1999a}
C.~McA. Gordon and J. Luecke, 
\textit{Seifert fibered surgeries on hyperbolic knots}, 
Abstracts Amer. Math. Soc. \textbf{20} (1999), 405. 


\bibitem{IchiharaJong2010}
K. Ichihara and I.~D. Jong, 
\textit{Toroidal Seifert fibered surgeries on Montesinos knots}, 
Comm. Anal. Geom. \textbf{18} (2010), no.3, 579--600.

\bibitem{IchiharaMasai}
K. Ichihara and H. Masai, 
\textit{Exceptional surgeries on alternating knots}, 
arXiv:1310.3472 (2013).

\bibitem{IchiharaMotegiSong2008}
K. Ichihara, K. Motegi, and H.-J. Song, 
\textit{{Seifert fibered slopes and boundary slopes on small hyperbolic knots}}, 
Bull. Nara Univ. Ed. Natur. Sci. \textbf{57} (2008), no.2, 21--25.

\bibitem{Menasco1984}
W. Menasco, 
\textit{Closed incompressible surfaces in alternating knot and link complements}, 
Topology \textbf{23} (1984), no.1, 37--44.

\bibitem{MiyazakiMotegi1997}
K. Miyazaki and K. Motegi, 
\textit{Seifert fibred manifolds and Dehn surgery}, 
Topology \textbf{36} (1997), no.2, 579--603.

\bibitem{Motegi2003}
K. Motegi, 
\textit{Dehn surgeries, group actions and Seifert fiber spaces}, 
Comm. Anal. Geom. \textbf{11} (2003), no.2, 343--389.

\bibitem{Oertel1984}
U. Oertel, 
\textit{Closed incompressible surfaces in complements of star links}, 
Pacific J. Math. \textbf{111} (1984), no.1, 209--230.

\bibitem{Patton1995}
R.~M. Patton, 
\textit{Incompressible punctured tori in the complements of alternating knots}, 
Math. Ann. \textbf{301} (1995), no.1, 1--22.

\bibitem{Perelman2002}
G. Perelman, 
\textit{The entropy formula for the Ricci flow and its geometric applications}, arXiv:math.DG/0211159 (2002).

\bibitem{Perelman2003}
G. Perelman, 
\textit{Ricci flow with surgery on three-manifolds}, 
arXiv:math.DG/0303109 (2003).

\bibitem{Perelman2003a}
G. Perelman, 
\textit{Finite extinction time for the solutions to the Ricci flow on certain three-manifolds}, 
arXiv:math.DG/0307245 (2003).

\bibitem{RolfsenBook}
D. Rolfsen, 
\textit{Knots and links}, 
Mathematics Lecture Series, no.7, Publish or Perish Inc., Berkeley, California, 1976.

\bibitem{Thurston1978}
W.~P. Thurston, 
\textit{The geometry and topology of $3$-manifolds}, 
Lecture notes, Princeton University, 1978, 
electronic version available at \texttt{http://www.msri.org/publications/books/gt3m}

\bibitem{Thurston1982}
W.~P. Thurston, 
\textit{Three-dimensional manifolds, Kleinian groups and hyperbolic geometry}, 
Bull. Am. Math. Soc. \textbf{6} (1982), no.3, 357--381.


\end{thebibliography}
\end{document}